\documentclass[12pt,reqno]{amsart}
\usepackage{amsthm,amsmath,amsfonts, amssymb,latexsym, url,cite,color}
\usepackage{srcltx}
\usepackage{hyperref,cleveref}
\hypersetup{colorlinks=true}
\numberwithin{equation}{section}
\usepackage{cite}
\usepackage{color}

\newcommand\exc{\mathop{ \rm exc}}
\newcommand\inv{\mathop{ \rm inv}}

\newcommand\cyc{\mathop{ \rm cyc}}

\newcounter{num}

\def\D{\mathcal{D}}

\def\N{\mathbb{ N}}
\def\R{\mathbb{ R}}
\def\S{\mathfrak{ S}}

\def\x{\mathbf{ x}}
\def\y{\mathbf{ y}}
\def\cyc{\mathop{\rm cyc}}
\def\rlm{\mathop{\rm rlm}}

\def\neg{\mathop{\rm neg}}
\def\nsum{\mathop{\rm nsum}}

\usepackage{epsfig}

\newtheorem{thm}{Theorem}
\newtheorem{lem}{Lemma}

\newtheorem{prop}{Proposition}
\newtheorem{cor}{Corollary}

\newtheorem{rem}{Remark}
\newtheorem{example}{Example}

\usepackage{booktabs}
\usepackage[dvipsnames]{xcolor}

\begin{document}

\title[Counting  derangements ]{Counting  derangements with signed right-to-left minima and excedances}
\author{Yanni Pei}
\address[Yanni Pei]{School of Mathematical Sciences,
 Dalian University of Technology, Dalian 116024, P. R. China}
\email{peiyanni@hotmail.com}

\author[Jiang Zeng]{Jiang Zeng}
\address[Jiang Zeng]{Univ Lyon, Universit\'e Claude Bernard Lyon 1, CNRS UMR 5208, Institut Camille Jordan, 43 blvd. du 11 novembre 1918, F-69622 Villeurbanne cedex, France}
\email{zeng@math.univ-lyon1.fr}
\date{\today}

\begin{abstract}
Recently  Alexandersson and Getachew proved some 
multivariate  generalizations of a 
 formula for enumerating signed excedances  in  derangements. In this paper  we first relate  their work to  a recent continued fraction  for  permutations  and  confirm  some of their observations. Our second main result is two refinements of 
their multivariate identities, which clearly explain  
the meaning  of each term  in their main  formulas.
 We
also  explore   some
similar  formulas  for permutations of type B.
\end{abstract}

\maketitle
\section{Introduction}
Over the last thirty years much work has been done  on enumeration over the symmetric group (and more generally  Coxeter or Weyl  groups), taking
into account also the sign (or one dimensional character) of each permutation, see 
\cite{DF92, Wa92, AGR05, Zh13, ELL21, KZ03, Mo15, Si16, AG21} and the references therein.  

In this paper we mainly study signed statistics over
derangements. Let $\sigma$ be   
a permutation  of 
the set  $[n]:=\{1,\cdots,n\}$, written as a word $\sigma=\sigma(1)\sigma(2)\cdots\sigma(n)$. Then 
$\sigma$  is called \emph{derangement} if $\sigma(i)\neq i$ for 
 $i\in [n]$.
The index  $i$ (resp. $\sigma(i)$) is an \emph{excedance index} (resp. value) if $\sigma(i)>i$.
Similarly, we say that $i$ (resp. $\sigma(i)$) is a \emph{right-to-left minimum  index} (resp. value) if $\sigma(i)<\sigma(j)$ for all $j>i$.  
Let  $\S_n$ (resp. $\D_n$) be the  set of permutations (resp. derangements) of $[n]$.
For $\sigma\in \S_n$, 
the sets of indices and values of  excedances and 
left-to-right minima  are denoted, respectively, by
\begin{align*}
\mathrm{EXCi(\sigma)}&=\{i\in [n] \mid 
  \textrm{and}\; \sigma(i)>i\};\\ 
\mathrm{EXCv(\sigma)}&=\{\sigma(i)\mid  i\in \mathrm{EXCi(\sigma)}\};\\ 
\mathrm{RLMi(\sigma)}&=\{i\in [n]\mid 
\sigma(i)<\sigma(j)\; \textrm{for}\;  j\in \{i+1, \ldots, n\}\};\\ 
\mathrm{RLMv(\sigma)}&=\{\sigma(i)\mid i\in \mathrm{RLMi(\sigma)}\}.
\end{align*}
Note that 
$\mathrm{EXCi(\sigma)}\cap \mathrm{RLMi(\sigma)}=\emptyset$ and 
	$\mathrm{EXCv(\sigma)}\cap \mathrm{RLMv(\sigma)}=\emptyset.$
Let $\exc(\sigma)$, $\rlm(\sigma)$ and $\cyc(\sigma)$ be respectively the number of excedances, right-to-left minima and cycles of the permutation $\sigma$.
It is well-known, see \cite{Ze93, Br00, KZ03}, that 
\begin{subequations}\label{eq:abc}
\begin{equation}\label{der:cyc:exc}
	\sum_{\sigma\in \D_n}(-1)^{\cyc(\sigma)}x^{\exc(\sigma)}=-(x+x^2+\cdots +x^{n-1}).
\end{equation}
As $\inv(\sigma)\equiv n-\cyc(\sigma)\pmod{2}$,
formula \eqref{der:cyc:exc} can be rephrased as 
\begin{equation}\label{der:inv:exc}
	\sum_{\sigma\in \D_n}(-1)^{\inv(\sigma)}x^{\exc(\sigma)}=(-1)^{n-1}(x+x^2+\cdots +x^{n-1}).
\end{equation}
\end{subequations}
which is in \cite[Proposition 4.3]{MR03}.
In this paper we shall freely switch from the $\inv$-version to $\cyc$-version of an identity and vice versa. 
In 2003 Ksavrelof and the second author \cite{KZ03} proved 
the following  refinement of \eqref{eq:abc}:
\begin{equation}\label{eq:kZ03}
\sum_{\sigma\in \D_{n},\; \sigma(n)=j}(-1)^{\cyc(\sigma)}x^{\exc(\sigma)}=-x^{n-j}\qquad (1\leq j\leq n-1).
\end{equation}
Similar  results  about excedances for Coxeter groups and colored permutation groups  have been  achieved  in \cite{BG05, Si11, Zh13, Si16}.

In 2021, 
Alexandersson and Getachew~\cite{AG21, AG21a} came up with two multivariate refinements of \eqref{eq:abc}:
\begin{subequations}\label{eq:AG}
\begin{align}\label{eq:AG1}
\sum_{\sigma\in \D_{n}}(-1)^{\inv( \sigma)}
\biggl( \prod_{j\in \mathrm{RLMv}(\sigma)}x_j \biggr) 
\biggl( \prod_{j\in \mathrm{EXCv}(\sigma)}y_j \biggr)  =(-1)^{n-1}\sum_{j=1}^{n-1}x_1\cdots x_j y_{j+1}\cdots y_n,
\end{align}
and 
\begin{align}\label{eq:AG2}
\sum_{\sigma\in \D_{n}}(-1)^{\inv( \sigma)}
\biggl( \prod_{j\in \mathrm{RLMi}(\sigma)}x_j \biggr) 
\biggl( \prod_{j\in \mathrm{EXCi}(\sigma)}y_j \biggr)  =(-1)^{n-1}\sum_{j=1}^{n-1}y_1\cdots y_j x_{j+1}\cdots x_n.
\end{align}
\end{subequations}
These identities promted us to seek for 
 possible multivariate analogues of \eqref{eq:kZ03}, 
 which would  refine  \eqref{eq:AG} consequently.  It is somehow surprising that 
our findings  do fulfill this requirement.
More precisely, for $j\in [n-1]$ our
two  multivariate analogues of \eqref{eq:kZ03}  read as follows:
\begin{subequations}\label{thm:PZ}
\begin{align}\label{thm:pz1}
	\sum_{\sigma\in \D_{n}, \, \sigma(n)=j}(-1)^{\cyc (\sigma)}
	\biggl( \prod_{i\in \mathrm{RLMv}(\sigma)}x_i \biggr) 
	\biggl( \prod_{i\in \mathrm{EXCv}(\sigma)}y_i \biggr)  =-x_1\cdots x_j y_{j+1}\cdots y_n,
\end{align}	
and 
\begin{align}\label{thm:pz2}
	\sum_{\sigma\in \D_{n}, \,\sigma(n+1-j)=1}(-1)^{\cyc (\sigma)}
	\biggl( \prod_{i\in \mathrm{RLMi}(\sigma)}x_i \biggr) 
	\biggl( \prod_{i\in \mathrm{EXCi}(\sigma)}y_i \biggr)  
	=- y_{1}\cdots y_{n-j}x_{n+1-j}\cdots x_n
\end{align}
\end{subequations}
from which  eqs. \eqref{eq:AG} follow by summing $j$ on $[n-1]$.

On the other hand,   Sokal and the second author \cite{SZ22} studied some multivariate Eulerian polynomials  whose ordinary generating functions have nice J-fractions  formulas. We note  that 
the specialization of the  
polynomial $\widehat{Q}_n$  in \cite[(2.29)]{SZ22}  
with $x_1=xy$, $x_2=x$, $u_2=v_1=y_1=1$, $u_1=v_2=y_2=y$,
$w_\ell=0$ turns out to be  the  enumerative  polynomial of derangements
\begin{align}\label{D_n(x,y)}
	D_n(x,y,\lambda):=\sum_{\sigma\in \D_n}\lambda^{\cyc (\sigma)}x^{\rlm(\sigma)}y^{\exc(\sigma)},
\end{align}
and Theorem 2.4 in  
\cite{SZ22} reduces to the following result
	\begin{subequations}
		\begin{align}\label{eq:J-fraction}
			\sum_{n=0}^\infty D_n(x,y, \lambda)t^n
			&=\cfrac{1}
			{1-\cfrac{\lambda xy\, t^2}{1-(x+y)\, t-\cfrac{(\lambda+1)(x+1)y\, t^2}{1-
						(x+2y+1)\, t-\cfrac{(\lambda+2)(x+2)y\, t^2}{\cdots}}}},
		\end{align}
		where the coefficients under the (n+1)-th row are $\gamma_0=0$ and for $n\geq 1$, 
		\begin{align}
			\gamma_{n}&=x+ny+n-1,\\
			\beta_{n}&=(\lambda+n-1)(x+n-1)y.
		\end{align}
	\end{subequations}
It is interesting to note that a special 
  case of \eqref{eq:AG1} does follow from \eqref{eq:J-fraction}. Indeed, 
setting $\lambda=-1$ in  \eqref{eq:J-fraction} yields 
\begin{equation}
	\sum_{n=0}^\infty D_n(x,y, -1)t^n=\frac{1-(x+y)t}{(1-xt)(1-yt)}=\frac{\alpha}{1-xt}+\frac{\beta}{1-yt}
\end{equation}
with $\alpha=-y/(x-y)$ and $\beta=x/(x-y)$,   equivalently,
\begin{equation}\label{spe:AG}
	D_n(x,y, -1)=\frac{xy}{x-y}(y^{n-1}-x^{n-1})=-\sum_{j=1}^{n-1}x^jy^{n-j}.
\end{equation}
which is  \eqref{eq:AG1} with 
$x_i=x$ and $y_i=y$ for $i\in [n]$.

This paper is organized as follows.  
In Section~\ref{laguerre}, we shall exploit  the J-fraction \eqref{eq:J-fraction} to confirm several observations in \cite{AG21} on 
$\mu_n(x):=D_n(x,1,1)$
and prove that the polynomials $\mu_n(x)$  are moments of 
\emph{co-recursive Laguerre  polynomials}. 
We then prove \eqref{thm:PZ} (see Theorem~\ref{thm:pz}) by constructing  {\it weight-preserving and  sign-reversing} 
(\textsc{wpsr}) involutions on derangements  in Section~\ref{type A}.
Finally, 
we establish two similar multivariate identities for permutations of  type B, which unify some results in \cite{AG21, Zh13} for permutations of both type A and type B in Section~\ref{type B}.


\section{Derangements with right-to-left minima and Laguerre polynomials}\label{laguerre}
Let  $d_{n,k}$ be  the number of derangements  in $\D_n$  with $k$ right-to-left minima, namely
\begin{align}\label{def:Dpoly}
\mu_n(x):=\sum_{\sigma\in \D_n}x^{\rlm(\sigma)}=\sum_{k=1}^{n} d_{n,k}x^k.
\end{align}
At the end of \cite{AG21} Alexandersson and  Getachew speculate some open problems about  the 
coefficients $d_{n,k}$.  
We note that  a more general multivariate polynomial 
of $\mu_n(x)$ had been studied in  a recent paper by Sokal and the second author~\cite{SZ22}. More precisely, 
the  
polynomial $\widehat{Q}_n$  in \cite[(2.29)]{SZ22}  
with $x_1=x_2=x$, $u_2=v_1=y_1=u_1=v_2=y_2=\lambda=1$ and $w_\ell=0$  becomes $\mu_n(x)$, therefore we derive 
from
\cite[Theorem 2.4]{SZ22} that
\begin{subequations}\label{cf:Drlm}
\begin{align}\label{cf:D}
	\sum_{n=0}^\infty \mu_n(x)t^n=\cfrac{1}
	{1-\cfrac{x\, t^2}{1-(x+1)\, t-\cfrac{2(x+1)\, t^2}{1-
				(x+3)\, t-\cfrac{3(x+2)\, t^2}{\cdots}}}},
\end{align}
where the coefficients under the (n+1)-th row are $\gamma_0=0$ and for $n\geq 1$, 
		\begin{align}
			\gamma_{n}&=x+2n-1,\\
			\beta_{n}&=n(x+n-1).
		\end{align}
	\end{subequations}
The above fraction generates  immediately  the first few values of $\mu_n(x)$ below (see \cite[Table~2]{AG21}).
\begin{align*}
\mu_2&=x\\
\mu_3&=x^{2}+x\\
\mu_4&=x^{3}+5 x^{2}+3 x\\
\mu_5&=x^{4}+11 x^{3}+21 x^{2}+11 x\\
\mu_6&=x^{5}+19 x^{4}+79 x^{3}+113 x^{2}+53 x\\
\mu_7&=x^{6}+29 x^{5}+211 x^{4}+589 x^{3}+715 x^{2}+309 x\\
\mu_8&=x^{7}+41 x^{6}+461 x^{5}+2141 x^{4}+4835 x^{3}+5235 x^{2}+2119 x\\
\mu_9&=x^{8}+55 x^{7}+883 x^{6}+6175 x^{5}+22357 x^{4}+43831 x^{3}+43507 x^{2}+16687 x
\end{align*}

Clearly, for $n\geq 2$ we have $d_{n,n-1}=1$ as
$\sigma=n 1 2 \cdots (n-1)$ is the only derangement $\sigma\in \D_n$ with  $\rlm(\sigma)= n-1$.  
The following recurrence  for $d_n:=|\D_n|$ is well-known, 
\begin{equation}\label{eq:d}
	d_n=(n-1)\left(d_{n-1}+d_{n-2}\right)\quad \textrm{with} \quad d_0=1,\quad d_1=0.
\end{equation}
This can be verified as in the following.
For $\sigma\in \D_n$ let $C=(n, \sigma(n), \sigma^2(n),\ldots, \sigma^{k}(n))$ be  the cycle containing $n$ with $k\geq 1$. 
If $k=1$, 
deleting the 2-cycle $C$  in $\sigma$ yields a derangement $\sigma'$  
of $[n-1]\setminus \{\sigma(n)\}$  
and there are $(n-1)d_{n-2}$ such derangements; 
if $k\geq 2$,  then replacing the cycle $C$ by $C^*=(\sigma(n), \sigma^2(n),\ldots, \sigma^{k}(n))$ in $\sigma$  yields  a derangement $\sigma^*$ of $[n-1]$ and 
there are $(n-1)d_{n-1}$ such derangements $\sigma$.

\begin{prop}
	For $n\geq 2$,
	let $\bar d_n=d_{n,1}$. 
	Then
	\begin{align}
		\bar d_n&=d_{n-1}+d_{n-2},\label{eq1}\\
		d_n&=(n-1)\bar d_{n},\label{eq2}\\
		\bar d_n&=(n-2)\bar d_{n-1}+(n-3)\bar d_{n-2} \qquad (n\geq 3)\label{eq3}. 
	\end{align}
\end{prop}
\begin{proof}
	Consider a $\sigma\in \D_n$ with  $\sigma(n)=1$.
	Then either $\sigma(1)=n$ or $\sigma(1)\neq n$. In the first case 
	the restriction of $\sigma$ on $\{2, \ldots, n-1\}$ is a derangement, in the second case define the permutation $\sigma'$ of $\{2, \ldots, n\}$  by
	$\sigma'(i)=\sigma(i)$ if $i\neq n$ and $\sigma'(n)=\sigma(1)$. Then 
	$\sigma'$ is a derangement  of $\{2, \ldots, n\}$. It follows that $\bar d_n=d_{n-1}+d_{n-2}$. Next we obtain \eqref{eq2} from \eqref{eq1} and \eqref{eq:d}. Finally we derive \eqref{eq3} from \eqref{eq1} and \eqref{eq2}.
\end{proof}
\begin{rem}
Eq. \eqref{eq3} was observed in \cite[Section 5.3]{AG21}, see \href{http://oeis.org/A000255}{A000255} \cite{Sl19}.
	As  $\sum_{n\geq 0}d_n \frac{x^n}{n!}=\frac{e^{-x}}{1-x}$,
	we derive from \eqref{eq1} the exponential generating function
\begin{align}\label{exp:gf}
		\sum_{n\geq 0} d_{n+2,1} \frac{x^n}{n!}&=\frac{e^{-x}}{(1-x)^2}.
	\end{align}
	By applying  the Stieltjes-Rogers addition formula to \eqref{exp:gf} one can derive the following J-fraction, which also follows from  \eqref{cf:D},
	\begin{align}\label{CF:dn1}
		\sum_{n\geq 0} d_{n+2,1} t^n=\cfrac{1}{1- t-\cfrac{2\, t^2}
			{1-3\, t-\cfrac{6\, t^2}{\cdots}}},
	\end{align}
	with $\gamma_0=1$, $\gamma_n=2n+1$ and $\beta_n=n(n+1)$ for $n\geq 1$.

\end{rem}

It was observed in \cite{AG21} that 
the sequence $d_{n, n-2}$ seems to be   \href{http://oeis.org/A028387}{A028387} in \cite{Sl19}. We provide a proof of this observation by manipulating the continued fraction \eqref{cf:D}, though there must be a simpler proof. 
\begin{prop}  We have $d_{3,1}=1$ and for $n\geq 3$,
	\begin{align}
		d_{n, n-2}&=(n-3)+(n-2)^2.\label{derbis}
	\end{align}
\end{prop}
\begin{proof}  In what follows
 we use  $[x^kt^n]\sum_{n\geq 0}a_n(x)t^n$ to denote the coefficient of 
 $x^kt^n$ in the formal series $\sum_{n\geq 0}a_n(x)t^n$ with $a_n(x)\in \R[x]$. 
 Hence $d_{n,n-2}=[x^{n-2}t^n]\sum_{n=0}^\infty \mu_n(x)t^n$ for $n\geq 3$.
 By \eqref{cf:D} we have 
	\begin{align*}
		d_{n,n-2}
		=&[x^{n-2}t^n]\cfrac{1}
		{1-\cfrac{x\, t^2}{1-(x+1)\, t-\cfrac{2(x+1)\, t^2}{1-
					(x+3)\, t}}}\\
		=&[x^{n-2}t^n]\frac{xt^2}{1-(x+1)t}+
		[x^{n-2}t^n]\frac{x^2t^4}{(1-(x+1)t)^2}\\
		&+
		[x^{n-3}t^{n-2}]\sum_{j\geq 1}\left(
		(x+1)\, t+\frac{2(x+1)\, t^2}{1-(x+3)\, t}\right)^j\\
		=&(n-2)+(n-3)+2\sum_{j=1}^{n-3}j
	\end{align*}
	which is equal to $(n-2)^2+n-3$. 
\end{proof}

\begin{rem}
	 The generating function is
	\begin{align}
	\sum_{n\geq 3} d_{n, n-2}x^{n-3}&=\frac{1+2x-x^2}{(1-x)^3}=1+5x+11x^2+19x^3+29 x^4+\ldots.\label{derbis:gf}
\end{align}	
	By \eqref{derbis} 
	we derive $
		d_{n+2, n}=d_{n+1, n-1}+2n$ $(n\geq 2)$
	It  should be possible
	to verify this recurrence directly. 
\end{rem}

\subsection{The polynomials $\mu_n(a)$'s as moments}
The Laguerre polynomials $L^{\alpha}_n(x)$ of degree $n$ is 
$$
L^{\alpha}_n(x)=\sum_{i=0}^n (-1)^i{n+\alpha\choose n-i}\frac{x^i}{i!}.
$$
The monic Laguerre polynomials $P^{\alpha}_n(x)=(-1)^n n!L_{n}^{\alpha}(x)$, whose weight function is $$x^\alpha e^{-x}, \quad x\geq 0,$$ 
may be  defined by 
the recurrence relation
\begin{equation}\label{recur:laguerre}
	P^{\alpha}_{n+1}(x)=(x-(\alpha+2n+1))
	P^{\alpha}_{n}(x)-n(n+\alpha)P^{\alpha}_{n-1}(x), \quad P^{\alpha}_{0}(x)=1,\quad P^{\alpha}_{-1}(x)=0.
\end{equation}
The \emph{associated Laguerre polynomials} 
$P^{\alpha}_{n}(x;c)$ are defined by the recurrence
\begin{equation}\label{recur:assolaguerre}
	P^{\alpha}_{n+1}(x;c)=(x-(\alpha+2n+2c+1))
	P^{\alpha}_{n}(x;c)-
	(n+c)(n+c+\alpha)P^{\alpha}_{n-1}(x;c)
\end{equation}
with $P^{\alpha}_{-1}(x;c)=0$ and 
$P^{\alpha}_{0}(x;c)=1$.
It is well known that the moments of Laguerre polynomials are the enumerative polynomials of permutations  by cycles or left-to-right minima. 
Kim and Stanton~\cite{KS15} described 
the moments of  $P^{\alpha}_{n}(x;c)$ using certain statistics on permutations and permutation tableaux. In the following we determine 
the orthogonal polynomials  whose 
moments are the enumerative polynomials  of derangments 
by right-to-left minima.
\begin{prop} 
The polynomials $p_n(x)$ orthogonal   with respect to the moment sequence 
	$\{\mu_n(a)\}$ are given  by 
	\begin{equation}\label{co-recur-laguerre}
		p_n(x)=(-1)^n n!L_n^{a-1}(x+1)+ (-1)^{n-1}(n-1)!(a-1)L_{n-1}^{a-1}(x+1;1),
	\end{equation}
	with $p_{-1}(x)=0$ and $p_{0}(x)=1$.
\end{prop}
\begin{proof}
	By \eqref{cf:Drlm} the orthogonal polynomials $P^*_n(x)$ with respect to the moments 
	$\{\mu_n(a)\}$ satisfy the recurrence, for $n\geq 1$,
	\begin{equation}\label{recur:lag*}
		P^*_{n+1}(x)=(x-(a+2n-1))P^*_{n}(x)-n(n+a-1)P^*_{n-1}(x) 
	\end{equation}
	with  $P^*_0(x)=1$, $P^*_1(x)=x$.
	By \eqref{recur:laguerre} and 
	\eqref{recur:assolaguerre} we verify straightforwardly that 
	the polynomials 
	\begin{align}\label{eq:poly}
	p_n(x):=
	P^{a-1}_{n}(x+1)+(a-1)P^{a-1}_{n-1}(x+1;1)
\end{align}
	 satisfy \eqref{recur:lag*} 
with  
	$$
	p_1(x)=(x+1-a)+(a-1)=x.
	$$
	This establishes \eqref{co-recur-laguerre}.
\end{proof}

\begin{rem}
One can also derive the above result from 
Chihara's theory  about the co-recursive  orthogonal polynomials~\cite{Ch57}, see also \cite{SS96}.
 The associated Laguerre polynomials have an explicit double sum formula, see \cite{KS15}.
\end{rem}

\section{Right-to-left minima and excedances in derangements}\label{type A}
We partition
$\D_n$ ($n\ge2$) according to the value of 
$n$ and the position of 1 in derangements, respectively, 
\begin{subequations}
\begin{align}
 \D_{n,j}&=\{\sigma\in \D_n: \sigma(n)=j\},\\
 \widetilde{\D}_{n,j}&=\{\sigma\in \D_n: \sigma(n+1-j)=1\}\qquad  (j\in [n-1]).
\end{align}
\end{subequations}
The following is the main theorem of  this section.
\begin{thm}\label{thm:pz}
For $n\ge2$ and $j\in [n-1]$, we have 
\begin{subequations}
\begin{align}\label{eq:pz1}
	\sum_{\sigma\in \D_{n,j}}(-1)^{\cyc (\sigma)}
	\biggl( \prod_{i\in \mathrm{RLMv}(\sigma)}x_i \biggr) 
	\biggl( \prod_{i\in \mathrm{EXCv}(\sigma)}y_i \biggr)  =-x_1\cdots x_j y_{j+1}\cdots y_n,
\end{align}
and
\begin{align}\label{eq:pz2}
	\sum_{\sigma\in \widetilde{\D}_{n,j}}(-1)^{\cyc (\sigma)}
	\biggl( \prod_{i\in \mathrm{RLMi}(\sigma)}x_i \biggr) 
	\biggl( \prod_{i\in \mathrm{EXCi}(\sigma)}y_i \biggr)  
	=- y_{1}\cdots y_{n-j}x_{n+1-j}\cdots x_n.
\end{align}
\end{subequations}
\end{thm}

We first give a  lemma which  reduces the proof of 
the above two identities \eqref{eq:pz1}
and \eqref{eq:pz2} by half. 
Note that the properties about $\mathrm{FIX}, \mathrm{EXC}$ and  $\mathrm{RLM}$ in the lemma appeared already  in  \cite[Lemma 21]{AG21}.

\begin{lem}\label{lem:key} The mapping 
$\zeta: \sigma\mapsto \rho\circ\sigma^{-1}\circ\rho$ on $\S_n$ with 
$\rho(i)=n+1-i$  for 
$i\in [n]$ is an involution on $\S_n$ such   that for $j\in [n]$,
\begin{align*}
\cyc(\sigma)&=\cyc(\zeta(\sigma)),\\
\sigma(n)=j&\Longleftrightarrow\zeta(\sigma)(n+1-j)=1,\\
j\in \mathrm{FIX}(\sigma)&\Longleftrightarrow \rho(j)\in \mathrm{FIX}(\zeta(\sigma)),\\
j\in \mathrm{EXCi}(\sigma)&\Longleftrightarrow \rho(j)\in \mathrm{EXCv}(\zeta(\sigma)),\\
j\in \mathrm{RLMi}(\sigma)&\Longleftrightarrow \rho(j)\in \mathrm{RLMv}(\zeta(\sigma)).
\end{align*}
Moreover,  
the restriction $\zeta:\D_{n, j}\to \widetilde{\D}_{n,j}$ 
(resp. $\zeta:\widetilde{\D}_{n,j}\to \D_{n, j}$)  is a bijection  for $j\in [n-1]$.  
\end{lem}
\begin{proof}
 For $\sigma\in \S_n$,   the two
permutations $\sigma$ and $\zeta(\sigma)$ are  conjugate  in $\S_n$,  so  they have the same cycle type and a fortiori  $\cyc(\sigma)=\cyc(\zeta(\sigma))$. Next, 
\begin{align*}
\zeta(\sigma)(n+1-j)=1
\Longleftrightarrow \sigma^{-1}(j)=n
\Longleftrightarrow \sigma(n)=j;
\end{align*}
\begin{align*}
\rho(j)\in \mathrm{FIX}(\zeta(\sigma))\Longleftrightarrow \zeta(\sigma)(\rho(j))=\rho(j)
\Longleftrightarrow \rho\circ \sigma^{-1}(j)=\rho(j)
\Longleftrightarrow \sigma(j)=j.
\end{align*}
For the fourth and fifth properties, as $(\zeta(\sigma))^{-1}(\rho(j))=\rho\circ \sigma(j)$, we have
\begin{align*}
\rho(j)\in \mathrm{EXCv}(\zeta(\sigma)))
\Longleftrightarrow \rho(j)>\rho(\sigma(j))\Longleftrightarrow j<\sigma(j)\Longleftrightarrow
j\in \mathrm{EXCi}(\sigma);
\end{align*}
\begin{align*}
j\in \mathrm{RLMi}(\sigma)&\Longleftrightarrow
\sigma(j)<\sigma(k) \quad \textrm{for} \quad k>j\\
&\Longleftrightarrow (\zeta(\sigma))^{-1}(\rho(j))>(\zeta(\sigma))^{-1}(\rho(k))\quad \textrm{for} \quad \rho(k)<\rho(j)\\
&\Longleftrightarrow 
 \rho(j)\in \mathrm{RLMv}(\zeta(\sigma)).
\end{align*}
The last property follows from the second and third properties.
\end{proof}

For convenience,
we  illustrate the  properties of mapping $\zeta$ by an  example.
\begin{example}
If $\sigma=(1,3,6,8,4)(5,9)(2)(7)$, then
 $\sigma^{-1}=(1, 4,8, 6,3)(5,9)(2)(7)$ and 
 $\zeta(\sigma)=(9,6, 2,4, 7)(5,1)(8)(3)$. Hence
 $$
 \sigma=326198745\in  \D_{9,5},\quad \zeta(\sigma)=543712986\in \widetilde{\D}_{9,5}.
 $$
 and 
 \begin{align*}
\mathrm{RLMi}(\sigma)&=\{4,8,9\},\qquad 
\mathrm{RLMv}(\zeta(\sigma))=\{1,2,6\},\\
\mathrm{EXCi}(\sigma)&=\{1,3,5,6\},\qquad 
\mathrm{EXCv}(\zeta(\sigma))=\{4,5,7,9\},\\
\mathrm{FIX}(\sigma)&=\{2,7\},\hspace{2cm}
  \mathrm{FIX}(\zeta(\sigma))=\{3,8\}. 
  \end{align*}
\end{example}

\begin{lem}\label{lem:equivalence}
The two identities \eqref{eq:pz1} and \eqref{eq:pz2} are equivalent.
\end{lem}
\begin{proof}
For $i\in [n]$, the mapping 
$\rho: i\mapsto n+1-i$ is a permutation of $[n]$. 
Now, making substitution 
$x_i\mapsto x_{\rho(i)}$ and $y_i\mapsto y_{\rho(i)}$ for $i\in [n]$ in \eqref{eq:pz2}, we obtain
\begin{align}\label{eq:rho}
\sum_{\sigma\in \widetilde{\D}_{n,j}}(-1)^{\cyc (\sigma)}
	\biggl( \prod_{i\in \mathrm{RLMi}(\sigma)}
	x_{\rho(i)} \biggr) 
	\biggl( \prod_{i\in \mathrm{EXCi}(\sigma)}y_{\rho(i)} \biggr)  =-x_1\cdots x_j y_{j+1}\cdots y_n.
\end{align}
By Lemma~\ref{lem:key},  the restriction of $\zeta$
on $\widetilde{\D}_{n,j}$ sets up a bijection from  
$\widetilde{\D}_{n,j}$ to $\D_{n,j}$ such that
\begin{align*}
j\in \mathrm{EXCi}(\sigma)&\Longleftrightarrow \rho(j)\in \mathrm{EXCv}(\zeta(\sigma)),\\
j\in \mathrm{RLMi}(\sigma)&\Longleftrightarrow \rho(j)\in \mathrm{RLMv}(\zeta(\sigma)).
\end{align*}
Therefore, letting  $\zeta(\sigma)=\tau$ for $\sigma\in \widetilde{\D}_{n,j}$,  eq.  \eqref{eq:rho} is equivalent to
$$
\sum_{\tau\in \D_{n,j}}(-1)^{\cyc (\tau)}
	\biggl( \prod_{i\in \mathrm{RLMv}(\tau)}
	x_{i} \biggr) 
	\biggl( \prod_{i\in \mathrm{EXCv}(\tau)}y_{i} \biggr)=-x_1\cdots x_j y_{j+1}\cdots y_n,
$$
which is exactly \eqref{eq:pz1}.
\end{proof}

For any derangement  $\sigma \in\D_{n}$, as 
 $\mathrm{EXCi(\sigma)}=[n]\setminus\mathrm{EXCv(\sigma^{-1})}$, 
 we derive from \eqref{eq:pz1} the following result.
\begin{cor}\label{conj1}
Let $\widehat{\D}_{n,j}=\{\sigma\in \D_n: \sigma(j)=n\}$ for $j\in [n-1]$. Then 
\begin{align}\label{eq:refineSi}
\sum_{\sigma\in \widehat{\D}_{n,j}}(-1)^{\cyc (\sigma)}
 \prod_{i\in \mathrm{EXCi}(\sigma)}x_i   
=-x_1\cdots x_j.
\end{align}
\end{cor}
Note that multiplying \eqref{eq:refineSi} by $t^j$ and  summing over $j\in [n-1]$ yields \cite[Theorem~7]{Si16}.

\subsection{Proof of Theorem~\ref{thm:pz}}
Thanks to Lemma~\ref{lem:equivalence} we only need to prove \eqref{eq:pz1}.
For $j\in \{1, 2, \ldots, n-1\}$ with $n\geq 3$, we divide $\D_{n,j}$ into two parts:
\begin{subequations}
\begin{align}\label{Udef}
U_{n,j}=\begin{cases}
\{(1,2, \ldots, n-1, n)\}& \textrm{if $j=1$},\\
\{\sigma\in\D_{n,j}\,|\,\sigma(1)\neq2,\, \sigma(2)=1\}
& \textrm{if $j\geq 2$},
\end{cases} 
\end{align}
and 
\begin{align}\label{Edef}
\overline{U}_{n,j}=\D_{n,j} \setminus U_{n,j}
\quad (1\leq j\leq n-1).
\end{align}
\end{subequations}

Let $\x=(x_k)_{k\geq 1}$ and $\y=(y_k)_{k\geq 1}$ be two sequences of variables.
For $\sigma\in\D_{n}$ define the weight
\begin{align}\label{weight:v}
 w(\sigma;\, {\x},{\y})=(-1)^{\cyc(\sigma)}
 \prod_{i\in \mathrm{RLMv}(\sigma)}x_i \prod_{i\in \mathrm{EXCv}(\sigma)}y_i.
\end{align}

\begin{lem}\label{lemma:key}
For $\sigma\in  \overline{U}_{n,j}$ with $j\in [n-1]$, let 
$\varphi(\sigma)=(\sigma(i_{\sigma}), \sigma(i_{\sigma}+1))\circ\sigma$, where $i_\sigma$ is  the smallest integer $i$ such that $\sigma(i)\neq i+1$.
Then $\varphi$ 
is a {\it weight-preserving and  sign-reversing} (\textsc{wpsr}) involution  on $\overline{U}_{n,j}$, i.e., 
$w(\sigma; \,\x,{\y})=-w(\varphi(\sigma);\,\x,\y)$.
\end{lem}
\begin{proof} Let $\sigma\in  \overline{U}_{n,j}$. Obviously permutations 
$\varphi(\sigma)$ and $\sigma$ have opposite  signs. 
To show  that $\varphi$ is an involution on $ \overline{U}_{n,j}$, 
we verify the following three  points.
\begin{itemize}
\item $\varphi(\sigma)\in \D_{n,j}$. If $j=1$, as 
 $\sigma\neq (1, 2, \ldots, n-1, n)$, there is an integer $i\in [n-2]$ such that $\sigma(i)\neq i+1$, so $1\leq i_\sigma\leq n-2$;
 if $j>1$, as $\sigma(n)=j$, we have  $\sigma(j-1)\neq j$, so $1\le i_\sigma \le j-1$. In any case, we have $n>i_\sigma+1$, so $\varphi(\sigma)(n)=\sigma(n)=j$. Besides  $\{\sigma(i_\sigma), \sigma(i_\sigma+1)\}\cap \{i_\sigma,i_\sigma+1 \}=\emptyset$.
Indeed, as $\sigma$ is a derangement, we have $\sigma(i)\neq i$ for any $i\in [n]$; 
if $i_\sigma=1$, then $\sigma(1)\neq2$ and $\sigma(2)\neq1$ because $\sigma\in \overline{U}_{n,j}$;
if $i_\sigma>1$, then  $\sigma(i_\sigma)\neq i_\sigma+1$ and $\sigma(i_\sigma+1)\neq i_\sigma$ because $\sigma (i_\sigma-1)=i_\sigma$. 
Hence, by (1) we have  
$\varphi(\sigma)\in\D_{n,j}$.
\item  $\varphi(\sigma)\notin U_{n,j}$. If $j=1$, then  $i_\sigma<n-1$ and 
$\varphi(\sigma)(i_\sigma)=\sigma(i_\sigma+1)\neq i_\sigma+1$.  Hence $\varphi(\sigma)\notin U_{n,1}$; 
If $j>1$ and 
$i_\sigma=1$, then $\varphi(\sigma)(2)=\sigma(1)\neq 1$,
if  $j>1$ and  $i_\sigma>1$, then 
$\sigma(1)\notin \{\sigma(i_\sigma), \sigma(i_\sigma+1)\}$,
so $\varphi(\sigma)(1)=\sigma(1)=2$ because $1<i_\sigma$. Hence  $\varphi(\sigma)\notin U_{n,j}$ for $j>1$.
\item $\varphi^2(\sigma)=\sigma$.
If  $i<i_\sigma$, then $\varphi(\sigma)(i)=\sigma(i)=i+1$ and $\varphi(\sigma)(i_\sigma)=\sigma(i_\sigma+1)\neq i_\sigma+1$.
Thus $i_\sigma=i_{\varphi(\sigma)}$ and 
$(\varphi(\sigma)(i_\sigma), \varphi(\sigma)(i_\sigma+1))=(\sigma(i_\sigma+1), \sigma(i_\sigma))$.
So $\varphi^2(\sigma)=\sigma$.
\end{itemize}
Next, we  show  that $\varphi$ is weight-preserving by  verifying the following two  points.
\begin{itemize}
\item $\mathrm{RLMv}(\sigma)=\mathrm{RLMv}(\varphi(\sigma))$:
If $j=1$, then  $\mathrm{RLMv(\sigma)}=\mathrm{RLMv(\varphi(\sigma))}=\{1\}$. So we consider the case $j>1$ in the following.
\begin{enumerate}
\item If $\sigma(i_\sigma)=1$, then the first $i_\sigma$ letters of $\sigma$ are $2, 3, \ldots, i_\sigma, 1$, so 
$\sigma(i_\sigma+1)>i_\sigma+1$ and there must exist an integer $k$ satisfied $k>i_\sigma+1$ and $\sigma(k)=i_\sigma+1$, thus  $\sigma(i_\sigma+1)\notin\mathrm{RLMv}(\sigma)$.
We have $\varphi(\sigma)(i_\sigma+1)=1\in\mathrm{RLMv}(\varphi(\sigma))$ and $\varphi(\sigma)(i_\sigma)=\sigma(i_\sigma+1)\notin\mathrm{RLMv}(\varphi(\sigma))$.
\item If $\sigma(i_\sigma)>1$ and  $\sigma(i_\sigma+1)=1\in\mathrm{RLMv(\sigma)}$,
there is  an integer  
$k>i_\sigma+1$ such that  $\sigma(k)=i_\sigma+1$, so  $\sigma(i_\sigma)\notin \mathrm{RLMv}(\sigma)$,
$\varphi(\sigma)(i_\sigma+1)=\sigma(i_\sigma)\notin\mathrm{RLMv}(\varphi(\sigma))$ and  $\varphi(\sigma)(i_\sigma)=1\in\mathrm{RLMv(\sigma)}$.
\item
If $\sigma(i_\sigma)>1$ and $\sigma(i_\sigma+1)>1$, then 
$1$ appears at the right of the letter  $\sigma(i_\sigma+1)$ 
in $\sigma$, hence  $\varphi(\sigma)(i_\sigma+1)\notin\mathrm{RLMv(\varphi(\sigma))}$.
\end{enumerate}

\item  $\mathrm{EXCv}(\sigma)=\mathrm{EXCv}(\varphi(\sigma))$:
As $\sigma(i_\sigma)\neq i_\sigma, i_\sigma+1$, it is clear that $\sigma(i_\sigma)>i_\sigma$ if and only if 
$\varphi(\sigma)(i_\sigma+1)=\sigma(i_\sigma)>i_\sigma+1$.
Similarly, we see that $\sigma(i_\sigma+1)>
i_\sigma+1$ if and only if 
$\varphi(\sigma)(i_\sigma)=\sigma(i_\sigma+1)>i_\sigma$. Hence
$\mathrm{EXCv(\sigma)}=\mathrm{EXCv(\varphi(\sigma))}$.
\end{itemize}
Summarizing we conclude that  $\varphi$ 
is the mapping we wished.
\end{proof}

Now, we are ready to complete  the proof of 
Theorem \ref{thm:pz}.
\begin{proof} [Proof of Theorem \ref{thm:pz}]
By Lemma~\ref{lemma:key}, for $j\in [n-1]$  we have 
\begin{align*}
P_{n,j}(\mathbf{x}, \mathbf{y})&:=\sum_{\sigma\in \D_{n,j}}w(\sigma;\,\mathbf{x}, \mathbf{y})=
\sum_{\sigma\in U_{n,j}}w(\sigma;\,{\x},{\y}).
\end{align*}
By \eqref{Udef}, the only element in $U_{n,1}$ is  $\sigma=2\,3\,\ldots n\,1$,  therefore 
\begin{align}\label{eq:j1}
P_{n,1}(\mathbf{x}, \mathbf{y})
=-x_1y_2\ldots y_n. 
\end{align}
 For $\sigma\in\D_{n}$ and $m\in \N$ define the  shifted version of \eqref{weight:v} by
\begin{equation}
w(\sigma; \delta^m({\x}),\delta^m({\y}))
:=(-1)^{\cyc(\sigma)}\prod_{i\in \mathrm{RLMv}(\sigma)}x_{i+m}
\prod_{i\in \mathrm{EXCv}(\sigma)}y_{i+m}.
\end{equation}
  If  $\sigma\in U_{n,j}$ with  $j>1$, then   $\sigma(2)=1$,  
let  $\psi(\sigma)=\sigma'(1)\sigma'(3)\ldots \sigma'(n)$ with $\sigma'(i)=\sigma(i)-1$ for
$i\in \{1, 3, \ldots, n\}$. 
It is easy to see that 
the  mapping $\psi:U_{n,j}\to \D_{n-1,j-1}$ is a bijection such  that
$w(\sigma;\, \x,\y)=x_1w(\sigma';\,\delta(\x),\delta(\y))$, thus
$$
P_{n,j}(\x,\y)=
x_1P_{n-1,j-1}(\delta(\x), \delta(\y))
$$
and by iterating, we obtain the equation
$$
P_{n,j}(\x,\y)=
x_1\ldots x_{j-1}P_{n-j+1,1}(\delta^{j-1}(\x), \delta^{j-1}(\y))
$$
from which follows  \eqref{thm:pz1}  by invoking \eqref{eq:j1}.
\end{proof}

\section{Identities  for permutations of type B}\label{type B}
Let $B_n$ be the set of permutations $\sigma$ of $[n]\cup 
\{-n,\ldots, -1\}$ such that $\sigma(-i)=-\sigma(i)$ for each $i\in [n]$. 
As usual, we denote the negative element $-i$ by $\overline{i}$, 
identify $\sigma$ with the word $\sigma=\sigma(1)\cdots\sigma(n)$ and 
write $|\sigma|:=|\sigma(1)|\ldots |\sigma(n)|$ the associated permutation of $[n]$.  We follow  Brenti~\cite{Br94} for definition of  permutation statistics of  type $B$.
The \emph{cycle decomposition} of $\sigma\in B_n$ is accomplished by first writing $|\sigma|$ as the disjoint union of cycles,
then turning each element $|\sigma(i)|$ to $\sigma(i)$.
Let $\cyc(\sigma)$ be the number of cycles of $|\sigma|$. 
For example, if $\sigma=\bar{6}42\bar{3}158\bar{7}$, then 
$|\sigma|=(1,\,{6},\,5)({3},\,2,\,4)({7},\,8)$ and 
 $\sigma=(1,\,\bar{6},\,5)(\bar{3},\,2,\,4)(\bar{7},\,8)$, hence $\cyc(\sigma)=3$.

\subsection{Excedance of type B}
If $\sigma\in B_n$, 
 the index 
$|\sigma(i)|$ with  $i\in [n]$ is a type \emph{B excedance} of $\sigma$
if $\sigma(|\sigma(i)|)=-|\sigma(i)|$ or $\sigma(|\sigma(i)|)>\sigma(i)$, and
the index $|\sigma(i)|$ is a type B \emph{anti-excedance} of $\sigma$
if $\sigma(|\sigma(i)|)=|\sigma(i)|$ or $ \sigma(|\sigma(i)|)<\sigma(i) $. Let
\begin{align*}
	\mathrm{EXC_B}(\sigma)&=\{|\sigma(i)|: \sigma(|\sigma(i)|)=-|\sigma(i)| \text{ or } \sigma(|\sigma(i)|)>\sigma(i), \; i\in [n]\},\\
	\mathrm{AnEXC_B}(\sigma)&=\{|\sigma(i)|:
	\sigma(|\sigma(i)|)=|\sigma(i)| \text{ or }
	\sigma(|\sigma(i)|)<\sigma(i), \; i\in [n]\}.
\end{align*}
For example, if $\sigma=\bar{6}42\bar{3}158\bar{7}\bar{9}=(1,\,\bar{6},\,5)(\bar{3},\,2,\,4)(\bar{7},\,8)(\bar 9)$, 
then 
\begin{align*}
\mathrm{EXC_B}(\sigma)&=\{1,2,7,9\},\\
\mathrm{AnEXC_B}(\sigma)&=\{3,4,5,6,8\}.
\end{align*}
\begin{rem} 
The above definition of excedance  is given by Brenti~\cite{Br94}, see another notion of excedance of type B in \cite{St94}.
\end{rem}
If $\sigma\in B_n$ 
let
${\rm Neg}(\sigma)=\{i\in [n] : \sigma(i)<0 \}$, and define the statistics
$$\neg(\sigma)=\#{\rm Neg}(\sigma), \quad
\nsum(\sigma)=\sum_{i\in {\rm Neg}(\sigma)}|\sigma(i)|.$$
We define the enumerative polynomial of a subset $T\subseteq B_n$ by
$$P_{T}(\x,s, t)=\sum_{\sigma\in T}w_{\sigma}(\x,s,t)
$$
with 
$$w_{\sigma}(\x,s, t)=(-1)^{\cyc(\sigma)}s^{\neg(\sigma)}t^{{\rm nsum}(\sigma)}\prod_{i\in\mathrm{EXC_B(\sigma)}}x_i.$$

Alexandersson and Getachew~\cite[Proposition 29]{AG21}
 proved the identity:
\begin{align}\label{eq:ags1}
	\sum_{\sigma\in \S_n}(-1)^{\cyc(\sigma)} \prod_{i\in \mathrm{EXCi}(\sigma)}x_i=
	-\prod^{n-1}_{j=1}(x_j-1).
\end{align}
The special $x_1=\cdots =x_n=x$ case of this
identity is well known, see \cite{Br00, KZ03} and 
 a type B analogue
of the latter case was given by 
Zhao \cite[Theorem 3]{Zh13}  as follows:
\begin{equation}\label{eq:zhao}
	P_{B_n}(\mathbf{x},1,1)=
	\begin{cases}
		-(x+1)(x-1)^{n-1} &\text{if $ n$ is odd};\\
		(x-1)^n & \text{if $n$ is even}.
	\end{cases}
\end{equation}
We can  unify  \eqref{eq:ags1} and \eqref{eq:zhao} by the  following theorem.
\begin{thm}\label{eq:Bexc} 
	Let $n\ge 1$, then we have
	\begin{align}\label{EXC}
		P_{B_n}(\x, s, t)=-\biggl(1+(-1)^{n-1}x_ns^nt^{\frac{n(n+1)}{2}}\biggr)\prod^{n-1}_{j=1}(x_j-1).
	\end{align}
\end{thm}

In order  to prove this identity we divide the set $B_n$ into three subsets:
\begin{align*}
	B^+_n&=\{\sigma\in B_n: \sigma(i)>0 \text{ for all } i\in [n]\},\\
	B^-_n&=\{\sigma\in B_n: \sigma(i)<0 \text{ for all } i\in [n]\},\\
	B^{\pm}_n&=\{\sigma\in B_n: \sigma(i)\sigma(j)<0 \text{ for some } i,j\in [n]\}
\end{align*}
and compute their enumerative polynomials respectively.

Clearly the subset $B^+_n$ is $\S_n$,  
so the corresponding identity is \eqref{eq:ags1}. 
We give an alternative  proof of \eqref{eq:ags1} by adapting the proof in \cite{KZ03} for a special case.

\begin{lem}[\cite{AG21}]\label{lem4}
	Let $n\ge1$, then
	\begin{align}\label{eq:lem4}
		P_{B^+_n}(\x,s,t)=-\prod^{n-1}_{j=1}(x_j-1).
	\end{align}
\end{lem}

\begin{proof}
Obviously the identity is true for $n=1, 2$.
Assume $n\ge 3$ and  divide $B^+_n$ into three subsets:
\begin{align*}
		B'_{n}&=\{\sigma\in B^{+}_n:\sigma(n-1)\neq n\text{ and } \sigma(n)\neq n \},\\
		B''_n&=\{\sigma\in B^{+}_n:\sigma(n)=n \},\\
		B'''_n&=\{\sigma\in B^{+}_n:\sigma(n-1)=n\}.
\end{align*}
It is easy to verify that the mapping 
$\varphi_1: \sigma \mapsto \sigma'=(\sigma(n-1),\sigma(n))\circ\sigma$ is a \textsc{wpsr} involution on $B'_{n}$, 
i.e., $w_\sigma(\x,s,t)=-w_{\sigma'}(\x,s,t)$;
the mapping $\varphi_2:\sigma\mapsto\sigma'=\sigma(1)\cdots\sigma(n-1)$ is 
a bijection from $B''_n$ to $B^+_{n-1}$ such that $w_\sigma(\x,s,t)=-w_{\sigma'}(\x,s,t)$, and the mapping
$\varphi_2\circ \varphi_1: \sigma\mapsto\sigma'$ is a bijection from $B'''_n$ to $B^+_{n-1}$ such that $w_\sigma(\x,s,t)=x_{n-1}w_{\sigma'}(\x,s,t)$.
Combining the three cases, 
 we have 
$$P_{B_n^+}(\x,s,t)=(x_{n-1}-1)P_{B_{n-1}^+}(\x,s,t).$$
As $ P_{B_2^+}(\x,s,t)=1-x_1 $, 
eq.~\eqref{eq:lem4}  follows by iterating the above recurrence.
\end{proof}

\begin{lem}\label{lem5}
	Let $n\ge1$, then
	\begin{align}\label{eq:lem5}
		P_{B_n^-}(\x,s,t)=
		(-1)^ns^nt^{n(n+1)/2}x_n\prod^{n-1}_{j=1}(x_j-1).
	\end{align}
\end{lem}

\begin{proof} 
	As $P_{B_n^-}(\x,s,t)=s^nt^{n(n+1)/2}P_{B_n^-}(\x,1,1)$, we assume that $s=t=1$ in the following. For $\sigma\in B^+_n$, the mapping 
	$\varphi: \sigma\to \bar\sigma$, where
	$\bar \sigma$ is the permutation defined by
	$\bar \sigma(i):=-\sigma(i)$ for $i\in [n]$, is a bijection from $B_n^+$ to $B_n^-$.
	It is clear that 
	$i\in \mathrm{AnEXC}_B(\sigma)$ if and only if $i\in \mathrm{EXC}_B(\bar \sigma)$.
	For $n\ge 1$, we have 
	\begin{align*}
		P_{B_n^-}(\x,1,1)&= \sum_{\sigma\in B^-_n}(-1)^{\cyc(\sigma)}\prod_{i\in \mathrm{EXC_B(\sigma)}}x_i\\
		&=\sum_{\sigma\in B^+_n }(-1)^{\cyc(\sigma)}\prod_{i\in \mathrm{AnEXC_B(\sigma)}}x_i\\
		&=x_1\cdots x_n\cdot\sum_{\sigma\in B^+_n}(-1)^{\cyc(\sigma)}\prod_{i\in\mathrm{EXC_B(\sigma)}}x_i^{-1}
	\end{align*}
	from which eq.~\eqref{eq:lem5}  follows from \eqref{eq:lem4}.
\end{proof}

%
For $a=n$ or $\bar n$, $k\in [n]$  and $\sigma\in B_{n-1}$ 
we define 
the permutation $\sigma'=\Psi(a,k, \sigma)$ in $B_n$ by 
 $
\sigma'(k)=a$, $ \sigma'(n)=\sigma(k)$, $ \sigma'(i)=\sigma(i)$
 for $i\in [n-1]\setminus\{ k\}$.
 Note that if $k=|a|=n$  the operation amounts to add the cycle 
 $(a)$ to $\sigma$.
 The mapping $\Psi$ will be  called the \emph{insertion operation} on $B_{n-1}$ in what follows. 
 For a permutation $\sigma\in B_n$, let
$\mathrm{Im}(\sigma)=\{\sigma(i): i\in [n]\}$.

%


\begin{lem}\label{lem6}
	Let $n\ge1$, then
	\begin{align}\label{eq:lem6}
		P_{B_n^\pm}(\x,s,t)=0.
	\end{align}
\end{lem}
\begin{proof}
By induction on $n\geq 1$,
we construct a \textsc{wpsr}  involution without fixed point on 
$B_n^\pm$.  For convenience, we shall describe such an involution by matching all the elements of $B_n^\pm$ two by two $\{\sigma, \sigma'\}$ such that  
$w_\sigma(\x,s,t)=-w_{\sigma'}(\x,s,t)$. 
 
Obviously $B_1^\pm=\emptyset$ and  $\{\{\bar{1}2, 2\bar{1}\}, \; \{1\bar{2}, \overline{2}1\}\}$ is a
\textsc{wpsr}  perfect matching of $B_2^\pm$, i.e.,
\begin{align*}
w_{\bar{1}2}(\x,s,t)+w_{2\bar{1}}(\x,s,t)&= x_1st-x_1st  =0,\\ 
w_{1\bar{2}}(\x,s,t)+w_{\overline{2}1}(\x,s,t)&= x_2st^2-x_2st^2=0.
\end{align*}
Hence $P_{B_2^\pm}(\x,s,t)=0$. Note that $\mathrm{Im}(\sigma)=\mathrm{Im}(\sigma')$.

Let $n\geq 3$.  It is easy to see that 
any $ \sigma\in B^{\pm}_n$ can be obtained 
by applying the insertion operation  $\Psi(a,k, \tau)$ with 
$a=n$ or $\bar n$,  $k\in [n]$  and $\tau\in B_{n-1}$. More precisely, there are four cases:
\begin{enumerate}
\item $a=\bar{n}$ and $\tau\in B^+_{n-1}$,  ($1'$) $a=n$ and  $\tau\in B^-_{n-1}$;
\item $a=\bar{n}$ and $\tau\in B^{\pm}_{n-1}$,  ($2'$) $a=n$ and $\tau\in B^{\pm}_{n-1}$.
\end{enumerate}

It suffices to show that the sum of weights of each of the four cases is zero. 
We just consider the cases (1) and (2) since the other two cases are similar.

\begin{enumerate}
\item	
Let $C_n$  be the set of permutations obtained by procedure (1). 
We show that there is such a \textsc{wpsr} perfect matching on $C_n$.
For  any $\sigma\in C_n$ let 
	$\sigma'=(\sigma(n-1),\sigma(n))\circ\sigma$. 
	It is easy to check that $\{\{\sigma,\sigma'\}\mid \sigma\in C_n\}$   is a \textsc{wpsr} perfect matching of $C_n$.
\item 		
Let $C'_n$ be the sets of permutations obtained by 
procedure (2). 
We show that there are such \textsc{wpsr} perfect matching on $C_n$ and $C'_n$.
Assume that there exists a \textsc{wpsr} perfect matching $M_{n-1}$  of $B^{\pm}_{n-1}$ and $(\tau,\tau')\in M_{n-1}$ is a doubleton such that 
$\mathrm{Im}(\tau)=\mathrm{Im}(\tau')$.
We claim that applying the insertion operation
 $\Phi$ on $C'_{n-1}$ yields a perfect matching of $C_n'$, where $k\in [n]$.
	
\begin{itemize}
\item If $k=n$, the operation amounts to add the cycle $(\overline{n})$ to $\tau$ and $\tau'$. 
So $\cyc(\sigma)=\cyc(\tau)+1$, $\cyc(\sigma')=\cyc(\tau')+1$, $\mathrm{EXC}_B(\sigma)=\mathrm{EXC}_B(\tau)\cup\left\{n\right\}$
and $\mathrm{EXC}_B(\sigma')=\mathrm{EXC}(\tau')\cup\left\{n \right\}$.\\

\item If $k\neq n$, it is equivalent to insert $\overline{n}$ in cycles $c=(\cdots \epsilon, \tau(k)\cdots)$ and $c'=(\cdots \epsilon, \tau'(k)\cdots)$ of $\tau$ and $\tau'$ as image of $\epsilon =k$ or $\bar k$, 
then we obtain the corresponding cycles
$(\cdots \epsilon, \overline{n}, \tau(k)\cdots)$ and $(\cdots \epsilon, \overline{n}, \tau'(k)\cdots)$
in $\sigma$ and $\sigma'$. 
Thus $\cyc(\sigma)=\cyc(\tau)$ and $\cyc(\sigma')=\cyc(\tau')$.
If $k\in\mathrm{EXC_B}(\tau)$, then 
$\mathrm{EXC}_B(\sigma)=\mathrm{EXC_B}(\sigma')=\mathrm{EXC}_B(\tau)\cup\{n\}\setminus\{k\}$.
If $k\notin\mathrm{EXC_B}(\tau)$, 
then $\mathrm{EXC}_B(\sigma)=\mathrm{EXC}_B(\sigma')=\mathrm{EXC}_B(\tau)\cup\{n\}$.
\end{itemize} 
As $\mathrm{Im}(\tau)=\mathrm{Im}(\tau')$ we have a \textsc{wpsr} perfect matching $M_n$ on $C'_n$.
\end{enumerate}

Combining  the above two cases  and the initial conditions, we derive \eqref{eq:lem6}.
\end{proof}

\begin{example}
Inserting $\overline{3}$ into matchings 
$\{\bar{1}2, 2\bar{1}\}$ and  $
\{1\bar{2}, \overline{2}1\}$
results 
\begin{equation*}
	\begin{array}{lr}
		w_{\bar{1}2\bar{3}}(\x,s,t)+w_{2\overline{1}\overline{3}}(\x,s,t)=-x_1x_3s^2t^4+x_1x_3s^2t^4=0,&\\
		w_{\overline{1}\overline{3}2}(\x,s,t)+w_{2\overline{3}\overline{1}}(\x,s,t)=x_1x_3s^2t^4-x_1x_3s^2t^4=0,&\\
		w_{\overline{3}\overline{1}2}(\x,s,t)+w_{\overline{3}2\overline{1}}(\x,s,t)=-x_3s^2t^4+x_3s^2t^4=0.&\
	\end{array}
\end{equation*}
\end{example}

%
\begin{proof}[Proof of Theorem~\ref{eq:Bexc}]
Clearly,  Theorem~\ref{eq:Bexc} follows from
Lemmas \ref{lem4} through \ref{lem6}. 
\end{proof}

\subsection{Right-to-left minima of type B}

Petersen \cite{Pe11} defined the set of \emph{right-to-left-minimum values} of $\sigma\in B_n$ by
\begin{align*}
	\mathrm{RLM_B}(\sigma)=\left\{ \sigma(i):0<\sigma(i)<|\sigma(j)| \text{ for all } j>i \right\}.
\end{align*}
Consider the enumerative polynomial
$$
Q_{B_n}(\y,s,t):=\sum_{\sigma\in B_n}w_{\sigma}(\y,s,t)
$$
where 
$$
w_{\sigma}(\y,s, t)=(-1)^{\cyc(\sigma)}s^{\rm neg(\sigma)}t^{{\rm nsum}(\sigma)}\prod_{j\in\mathrm{RLM_B(\sigma)}}y_j.
$$
The following is a right-to-left minima analogue of Theorem~\ref{eq:Bexc}.
\begin{thm}\label{thm:rlmB}
	For $n\ge1$,  we have
	\begin{align}\label{eq:RLM_B}
		Q_{B_n}(\y,s,t)=(-1)^n\prod_{i\in [n];\,  i\; \text{even}}(y_i-1)\prod_{i \in [n];\, i \;\text{odd}}(y_i+st^i).
	\end{align}
\end{thm}
\begin{proof}
The formula is obvious  for $n=1$, i.e.,  $Q_{B_1}(\y,s,t)=-(y_1+st)$.
For $n\ge2$, partition $B_n$ in two subsets,
\begin{align*}
	B'_n&=\{\sigma\in B_n:|\sigma(n)|=n \},\\
	B''_n&=\{\sigma\in B_n:|\sigma(n)|\neq n  \}.
\end{align*}

For $\sigma\in B'_n$, 
let  $\sigma'$ be the restriction of $\sigma$ on $[n-1]$.  
Then $\phi_1: \sigma\mapsto \sigma'$ is a bijection 
 from $B'_n$ to $\{n, \bar n\}\times B_{n-1}$ satisfying 
  $$
  w_\sigma(\y,s,t)=\begin{cases}
  -y_nw_{\sigma'}(\y,s,t)&\textrm{if $\sigma(n)=n$},\\
  -st^n w_{\sigma'}(\y,s,t)&\textrm{if $\sigma(n)=\bar n$}.
  \end{cases}
  $$
  It follows that
\begin{align}\label{eq:B'}
	\sum_{\sigma\in B'_n }w_\sigma(\y,s,t)
	=-(y_n+st^n)Q_{B_{n-1}}(\y,s,t).
\end{align}
Let  
$$
C_n^k=\{\sigma\in B''_n: |\sigma(2k-1)|=n\},
$$
where $k\in \{1, \ldots, (n-1)/2\}$ if $n$ is odd and 
$k\in \{1, \ldots, (n-2)/2\}$ if $n$ is even.
Then the mapping $\sigma\mapsto (\sigma(2k-1), \sigma(2k))\circ \sigma$ is a \textsc{wpsr} involution on $C_n^k$. Therefore
\begin{align}\label{eq:BE}
\sum_{\sigma\in B_n''}w_\sigma(\y,s,t)=
\sum_{\sigma\in E_n}w_\sigma(\y,s,t)
\end{align}
with $E_n=\{\sigma\in B''_n: |\sigma(n-1)|=n\}$ if $n$ is even and $E_n=\emptyset$ if $n$ is odd.

If $n$ is even, 
for $\sigma\in E_n$ let  $\sigma'$ be the restriction of $\sigma$ on $[n]\setminus \{n-1\}$.  
Then $\phi_2: \sigma\mapsto \sigma'$ is a bijection 
 from $E_n$  to $\{n, \, \bar n\}\times B_{n-1}$ satisfying 
  $$
  w_\sigma(\y,s,t)=\begin{cases}
  w_{\sigma'}(\y,s,t)&\textrm{if $\sigma(n-1)=n$},\\
  st^n w_{\sigma'}(\y,s,t)&\textrm{if $\sigma(n-1)=\bar n$}.
  \end{cases}
  $$
Hence, if $n$ is even,
then
\begin{align}\label{eq:E}
\sum_{\sigma\in E_n}w_\sigma(\y,s,t)=(1+st^n)\cdot Q_{B_{n-1}}(\y,s,t). 
\end{align}
It follows from \eqref{eq:B'}, \eqref{eq:BE} and \eqref{eq:E} that
$$ 
Q_{B_n}(\y,s,t)=
\begin{cases}
-(y_n+st^n)\cdot Q_{B_{n-1}}(\y,s,t)&\textrm{if $n$ is odd};\\
-(y_n-1)\cdot Q_{B_{n-1}}(\y,s,t)&\textrm{if $n$ is even}.
   \end{cases}
   $$
Invoking the initial value of $Q_{B_{1}}(\y,s,t)$, 
this recurrence implies \eqref{eq:RLM_B}.
\end{proof}
\begin{rem} When $s=0$ eq. \eqref{eq:RLM_B} reduces to~\cite[Corollary 22]{AG21a}, see also \cite[Corollary 36]{AG21}:
\begin{align}\label{ag21:RLMv_A}
	\sum_{\sigma\in\S_n}(-1)^{\cyc(\sigma)}\prod_{i\in\mathrm{RLMv}(\sigma)}y_i=(-1)^n \prod_{j\in [n]\atop
		\textrm{j odd}} y_j\cdot \prod_{j\in [n]\atop\textrm{j even}} (y_j-1)
\end{align}
which  is also the $q=-1$ case of a formula 
due to  Bj\"orner and Wachs~\cite[(5.3)]{BW91}, 
\begin{align}\label{Sn-xs}
	\sum_{\sigma\in\S_n}q^{\inv (\sigma)}
	\prod_{j\in\mathrm{RLMv}(\sigma)}y_j
	=\prod_{i=1}^{n}(y_i+q+\cdots+q^{i-1}).
\end{align}
Note that Poznanovi\'c \cite{Po14} and Eu et al. \cite{ELW15}
have given far-reaching generalizations of  \eqref{Sn-xs} to 
 other types of permutations  and  colored permutations.
\end{rem}

\section{Acknowledgements}
The first author was
supported by the \emph{China Scholarship Council}.
This work was done during
her  visit  at  Universit\'e
Claude Bernard Lyon 1 in 2021-2022.


\begin{thebibliography}{99}

\bibitem[AGR05]{AGR05}
R. M. Adin, I. M. Gessel, Y. Roichman, 
Signed Mahonians, 
J. Combin. Theory Ser. A 109(2005) 25--43.

\bibitem[AG21a]{AG21a}
P. Alexandersson, F. Getachew,
An involution on derangements preserving excedances and right-to-left minima,
S\'em. Lothar. Combin. 86B(2022), Art. 14, 9 pp. 

\bibitem[AG21]{AG21}
P. Alexandersson, F. Getachew,
An involution on derangements preserving excedances and right-to-left minima,
arXiv preprint \url{arXiv:2105.08455}.

\bibitem[BG05]{BG05}
E. Bagno, D. Garber, 
On the excedance number of colored permutation groups,
 S\'em. Lothar. Combin. 53(2004/2006), Article B53f.


\bibitem[BW91]{BW91}
A. Bj\"orner, M. L. Wachs,
Permutation statistics and linear extensions of posets,
J. Combin. Theory Ser. A 58(1991) 85--114.



\bibitem[Br94]{Br94}
F. Brenti,
$q$-Eulerian polynomials arising from Coxeter groups,
European J. Combin. 15(1994) 417--441.

\bibitem[Br00]{Br00}
F. Brenti, 
A class of $q$-symmetric functions arising from plethysm, 
J. Combin. Theory Ser. A 91(2000) 137--170.


\bibitem[Chi57]{Ch57}
T. S. Chihara, 
On co-recursive orthogonal polynomials,
Proc. Amer. Math. Soc. 8(1957) 899--905.

\bibitem[DF92]{DF92}
J. D\'esarm\'enien, D. Foata,
The signed Eulerian numbers.
Discrete Math. 99(1992), no. 1-3, 49--58.

\bibitem[ELL21]{ELL21}
S. P. Eu,  Z. C. Lin, Y. H. Lo,
Signed Euler-Mahonian identities. 
European J. Combin. 91(2021), Paper No. 103209.

\bibitem[ELW15]{ELW15}
S. P. Eu, Y. H. Lo, T. L. Wong, 
The sorting index on colored permutations and even-signed permutations,
Adv. in Appl. Math. 68(2015) 18--50.

\bibitem[KZ03]{KZ03}
G. Ksavrelof, J. Zeng, 
Two involutions for signed excedance numbers, 
S\'em. Lothar. Combin. 49(2002/04), Art. B49e.


\bibitem[KS15]{KS15}
J. S. Kim,  D. Stanton,  
The combinatorics of associated Laguerre polynomials, 
SIGMA Symmetry Integrability Geom. Methods Appl. 11(2015), Paper 039.
 
\bibitem[MR03]{MR03}
R. Mantaci, F. Rakotondrajao, 
Exceedingly deranging, 
Adv. in Appl. Math., 30(2003) 177--188.

\bibitem[Mo15]{Mo15}
P. Mongelli,  Signed excedance enumeration in classical and affine Weyl groups. J. Combin. Theory Ser. A 130 (2015), 129--149.

\bibitem[Pe11]{Pe11}
T. K. Petersen,  
The sorting index,
Adv. in Appl. Math., 47(2011) 615--630.

%
\bibitem[Si11]{Si11}
S. Sivasubramanian,  
Signed excedance enumeration via determinants,
Adv. in Appl. Math., 47 (2011), no. 4, 783--794.


\bibitem[Si16]{Si16}
S. Sivasubramanian,  
Enumerating excedances with linear characters in classical Weyl groups,
S\'em. Lothar. Combin. 74([2015-018]), Art. B74c.


\bibitem[Po14]{Po14}
S. Poznanovi\'c,
The sorting index and equisdistribution of set-valued statistics over permutations,
J. Combin. Theory Ser. A 125(2014) 254--272.

\bibitem[SS96]{SS96}
R. Simion, D. Stanton, 
Octabasic Laguerre polynomials and permutation statistics. 
J. Comput. Appl. Math. 68 (1996), no. 1-2, 297--329.

\bibitem[Sl19]{Sl19}
N. J. A. Sloane,
The On-Line Encyclopedia of Integer sequences. 2019. URL: \url{https://oeis.org/}

\bibitem[SZ22]{SZ22}
A. Sokal, J. Zeng,
Some multivariate master polynomials for permutations, set partitions, and perfect matchings, and their continued fractions,
Adv. in Appl. Math., 138, 102341 (2022).

\bibitem[St94]{St94}
E. Steingr\'{\i}msson, 
Permutation statistics of indexed permutations. (English summary)
European J. Combin. 15(1994), no. 2, 187--205.

\bibitem[Wa92]{Wa92}
M. L. Wachs, 
An involution for signed Eulerian numbers, 
Discrete Math. 99(1992) 59--62.

\bibitem[Zh13]{Zh13}
A. F. Y. Zhao,  
Excedance numbers for the permutations of type B,
Electron. J. Combin. 20(2013) Paper 28. 

\bibitem[Ze93]{Ze93}
J. Zeng,  
{\'{E}num\'{e}rations de permutations et {$J$}-fractions continues}. 
European J. Combin. 14 (1993), no. 4, 373--382.
\end{thebibliography}
\end{document}